\documentclass{amsart}
\usepackage{amssymb}
\usepackage{amscd}
\usepackage{amsmath}
\usepackage{enumerate}
\usepackage[all]{xy}
\usepackage{amsfonts}
\usepackage{latexsym}
\usepackage[pdftex]{graphicx}
\newtheorem{theo}{Theorem}[section]
\newtheorem{lemma}[theo]{Lemma}
\newtheorem{prop}[theo]{Proposition}
\newtheorem{cor}[theo]{Corollary}

\newtheorem{example}[theo]{Example}

\newtheorem{remark}[theo]{Remark}

\def\bC{{\mathbb C}}

\def\bR{{\mathbb R}}
\def\bZ{{\mathbb Z}}
\def\bT{{\mathbb T}}

\def\e{{\epsilon}}

\def\a{{\alpha}}
\def\b{{\beta}}

\newcommand{\fgbar}{f\bar{g}}

\newcommand{\B}{\mathbb{B}}
\newcommand{\BS}{\mathbb{S}}
\newcommand{\BD}{\mathbb{D}}
\newcommand{\s}{\vspace{0,3cm}}

\begin{document}
\title[Boundaries of Milnor Fibres]{The boundary of the Milnor fibre of complex and real analytic non-isolated singularities}
\author{Javier Fern\'andez de Bobadilla}
\author{Aur\'elio Menegon Neto}
\address{ICMAT. CSIC-Complutense-Aut\'onoma-Carlos III}
\email{javier@mat.csic.es}
\address{Departamento de Matem\'atica - Universidade Federal da Para\'iba}
\email{aurelio@mat.ufpb.br}
\thanks{Research partially supported by the ERC Starting Grant project TGASS and by Spanish Contract MTM2007-67908-C02-02. The authors thank the Facultad de Ciencias Matem\'aticas of the Universidad Complutense de Madrid for the excellent working conditions. The second named author received financial support from CONACYT, Mexico; CNRS, France through the Laboratorio Internacional Solomon Lefschetz and an ECOS-ANUIES agreement and FAPESP, Brasil, as well as through grants from CONACYT and from the Posgrado en Matem\'aticas of Universidad Nacional Aut\'onoma de M\'exico}
\date{24-3-2013}
\begin{abstract}
Let $f$ and $g$ be holomorphic function-germs vanishing at the origin of a complex analytic germ of dimension three. Suppose that they have no common irreducible component and that the real analytic map-germ $\fgbar$
has an isolated critical value at $0$. We give necessary and sufficient conditions for the real analytic map-germ $\fgbar$ to have a Milnor fibration and we prove that in this case the boundary of its Milnor fibre is a 
Waldhausen manifold. As an intermediate milestone we describe geometrically the Milnor fibre of map-germs of the form $\fgbar$ defined in a complex surface germ, and we prove an A'Campo-type formula for the zeta function of its monodromy. 
\end{abstract}

\maketitle

\section{Introduction}

It is classically known that there is a rich interplay between $3$-manifold
theory and the topology of isolated singularities in complex surfaces. This
goes back to the work of F. Klein by the end of the  19th century, and then
made clearer by in the  early 1960s, by the work of Grauert \cite{Gr} and
Mumford \cite{Mu}. A closed oriented $3$-manifold $M$  is the link of some isolated complex surface singularity $(V,p)$ if and only if $M$ is a Waldhausen manifold (a $3$-manifold which is obtained by gluing some Seifert manifolds along tori) with negative definite intersection form.

This important theorem has played, on one hand, a key role for understanding the topology of surface singularities through the work of W. Neumann and many others. On the other hand, the links of 
isolated surface singularities provide a very interesting class of $3$-manifolds which, thanks to their rich algebraic nature, have proved to be rather useful for $3$-manifold theory, as for instance for 
the understanding of the Casson invariant, Seiberg-Witten invariants, Floer homology and other important
invariants of $3$-manifolds that  have been discovered in the last decades.

In this sense, it is interesting to find new classes of $3$-manifolds, besides
the links of isolated complex surface singularities, that have a rich (possibly algebraic) geometric structure. One has the
theorem of A. Nemethi, A. Szilard, F. Michel and A. Pichon, stating that if $f$ is a holomorphic function-germ $(\bC^3,0) \to (\bC,0)$ with a $1$-dimensional critical set, 
then the boundary of the Milnor fiber is a Waldhausen manifold. The theorem was announced first by F. Michel and A. Pichon \cite{MP}, but its proof contained a gap. In~\cite{MPW} F. Michel,
A. Pichon and C. Weber provided a proof valid for some classes of non-isolated singularities. The first complete proof of the general (reduced holomorphic) case of the theorem was provided by A. Nemethi and 
A. Szilard
in \cite{NS}, where even an algorithm to compute the graph describing the Waldhausen manifold is given.
Shortly afterwards F. Michel and A. Pichon have provided in \cite{MP3} another complete proof which is more in the spirit of the original method they proposed. In this paper we envisage a similar problem, although in a more general approach. 

Let $(X,0)$ be a $3$-dimensional isolated singularity complex analytic germ at the origin. Let $f$ and $g$ be holomorphic function-germs defined in $(X,0)$. We say that the real analytic map-germ $\fgbar: (X,0) \to (\bC,0)$ has a Milnor fibration if there exist small positive real numbers $\e_j$ and $\eta$ with $0<\eta<<\e_j <<1$ such that the restriction
$$(\fgbar)|_{(\fgbar)^{-1}(\BD_\eta^*) \cap X \cap \Delta_{\underline\e}}: (\fgbar)^{-1}(\BD_\eta^*) \cap X \cap \Delta_{\underline\e} \to \BD_\eta^*$$
is the projection of a locally trivial fibration, where $\Delta_{\underline\e}$ is the polydisk in the ambient space $\bC^N \supset X$ defined by
$$\Delta_{\underline\e}:= \{ (z_1, \dots, z_N) \in \bC^N: |z_1|\leq\e_1, \dots, |z_N|\leq \e_N \},$$
$\BD_\eta$ is the closed ball around $0$ in $\bC$ with radius $\eta$ and $\BD_\eta^*:= \BD_\eta \backslash \{0\}$.

A. Pichon and J. Seade proved in \cite{PS} that if $f$ and $g$ are holomorphic function-germs defined on a holomorphic germ of surface $X$ with no common irreducible components and if the real analytic map-germ $\fgbar: (X,0) \to (\bC)$ has an isolated critical value, then $\fgbar$ has the Thom $a_f$ condition and therefore it has a Milnor fibration in the tube. They also conjectured that this was true in higher dimensions and latter they gave a proof to this conjecture in \cite{PS2}, which turned out to be false, as pointed by A. Parusinski with a counter-example (see \cite{PS3}).

Our main theorem is:

\begin{theo} \label{theo_main}
Let $(X,0)$ be a $3$-dimensional isolated singularity complex analytic germ at the origin. Let $f$ and $g$ be holomorphic function-germs (not necessarily reduced) defined in $(X,0)$ such that 
$V(f):=f^{-1}(0)$ and $V(g):=g^{-1}(0)$ have no common irreducible components and such that the real analytic map-germ $\fgbar: (X,0) \to (\bC,0)$ has an isolated critical value. Let $\Sigma$ denote the singular set
of the surface $V(fg)$, considered with reduced structure. Then:

\begin{itemize}
\item[$(i)$] For each irreducible component $\Sigma_i$ of the intersection $V(f) \cap V(g)$, let $H_i \subset X$ be a generic transversal section, with $i= 1, \dots, k$. 
Then $\fgbar$ has a Milnor fibration in the tube if and only if 
the restriction of $\fgbar$ to each $H_i$ has isolated critical values (which is equivalent to an easy to check numerical criterion at the level of the 
embedded resolution of the restriction of $fg$ to $H_i$, as showed in \cite{PS}). 

\item[$(ii)$] In this case, the boundary of the Milnor fibre $L_t := (\fgbar)^{-1}(t) \cap X \cap \BS_\e$, for $t \in \BD_\eta$, is a Waldhausen manifold.
\end{itemize}
\end{theo}

In $(i)$ of the theorem above we give necessary and sufficient conditions for $\fgbar$ defined in higher dimensions to have a Milnor fibration in the tube. We prove this in section \ref{section_3}.

The proof of $(ii)$ of Theorem \ref{theo_main} provides a shorter proof of the theorem of Michel-Pichon-Weber and Nemethi-Szilard (for the reduced holomorphic case) and generalizes to real analytic map-germs of the form $\fgbar$, even when either $f$ or $g$ is not reduced. The main new technique is a careful study of the vertical monodromies associated to the irreducible components of the singular set, by means of a careful dissection of the transversal Milnor fibre. The proof is based in a detailed understanding of the Milnor fibre of a map-germ of the form $\fgbar$ defined in the complex plane, in terms of an embedded resolution of $\{fg=0\}$. 

Precisely, in section \ref{section_2} we carefully dissect the Milnor fibre of a map-germ $\fgbar: (X,0) \to (\bC,0)$ (defined on a complex analytic germ of surface) into simple pieces, guided by the structure of an embedded resolution of the complex curve $fg$, in such a way that each piece is either a finite disjoint union of cylinders or a finite covering over a multi-punctured disk.

Then in section \ref{section_4} we proceed to the $3$-dimensional case by means of a slicing argument. We first split the boundary of the Milnor fibre of $\fgbar: (X, 0) \to (\bC, 0)$ into two parts:
one part that is inside a tubular neighbourhood of the link of the critical set of $\fgbar$ in the Milnor polydisc (called the vanishing zone), and other part that is away of this neighbourhood.
We first show that this second part of the Milnor fibre is a Waldhausen manifold, and since these two parts glue together along a finite union of tori, it just remains to show that the part of the Milnor fibre
of $\fgbar$ which lies inside the vanishing zone is a Waldhausen manifold. We do so by means of a slicing argument: we describe it as a fibre bundle over a circle whose fibre is the Milnor fibre of the restriction of
$\fgbar$ to a hyperplane section of $X$, with the decomposition given in section \ref{section_2}, and whose monodromy respects such decomposition.

The description of the Milnor fibre of a germ of the form $\fgbar$ defined in a complex surface $Y$ (section \ref{section_2}) also allows us to generalize (for real analytic map-germs of the type $\fgbar$ defined
in a complex surface) the A'Campo formula for computing the zeta-function of the monodromy, in terms of an embedded resolution of singularities. This is Theorem~\ref{funczeta} of section \ref{section_5}.

A natural next step would be to give an algorithm to determine the graph associated to the boundary of the Milnor fibre of functions $f\bar g$ generalizing the algorithm of~\cite{NS} to this kind of real analytic germs.

The authors thank A. Pichon and J. Seade, who suggested the problem and held useful conversations with us, and specially to A. Pichon who detected a gap in a previous version of this paper. We also acknowledge A. Nemethi for a fruitful discussion.


\section{A description of the Milnor fibre of $\fgbar$ defined in a germ of complex surface in terms of its embedded resolution}
\label{section_2}

Let $(X,0)$ be a germ of isolated singularity complex surface at the origin  and let $f$ and $g$ be holomorphic function-germs defined in $(X,0)$ such that $V(f):=f^{-1}(0)$ and $V(g):=g^{-1}(0)$ have no common irreducible components and such that the real analytic map-germ $\fgbar: (X,0) \to (\bC,0)$ has an isolated critical value. We can define $F_\eta$ to be its Milnor fibre, as discussed in the previous section.

Let $\pi: \tilde{M} \to \B_\e$ be a minimal embedded resolution of the complex curve $fg$. Then consider:
\begin{itemize}
\item[$\bullet$] $\pi^{-1} (0) := E = \bigcup_{i=1}^s E_i$, the exceptional divisor with its decomposition into irreducible components;
\item[$\bullet$] $(\fgbar \circ \pi)^{-1}(0) = (fg \circ \pi)^{-1}(0) = \sum_{i=1}^s k_i E_i + \tilde{C}$ the total transform, where $\tilde{C}$ is the strict transform, which has a decomposition into connected components $\tilde{C}= \bigcup_{p=1}^w \tilde{C_p}$;
\item[$\bullet$] $F_\eta\cong (\fgbar \circ \pi)^{-1}(\eta)$, the Milnor fibre of $\fgbar$.
\end{itemize}

For each $i$ with $1\leq i \leq s$ let $U_i$ be a closed tubular neighbourhood of $E_i$ in $\tilde{M}$, and for each $p$ with $1\leq p \leq w$ let $\tilde{U_p}$ be a closed tubular neighbourhood of $\tilde{C_p}$ in $\tilde{M}$. Then we define the following closed sets (see figure \ref{figT10}):
\begin{itemize}
\item[$\bullet$] $V_{ij}:= U_i \cap U_j$;
\item[$\bullet$] $\tilde{V}_{ip}:= U_i \cap \tilde{U}_p$;
\item[$\bullet$] $V_i := U_i \backslash \big( \bigcup_{\substack{i,j=1 \\ i \neq j}}^s V_{ij} \cup \bigcup_{p=1}^w \tilde{V}_{ip} \big)$;
\item[$\bullet$] $\tilde{V}_p := \tilde{U}_p \backslash \bigcup_{i=1}^s \tilde{V}_{ip}$.
\end{itemize}

\begin{figure}[!h] 
\centering 
\includegraphics[scale=0.5]{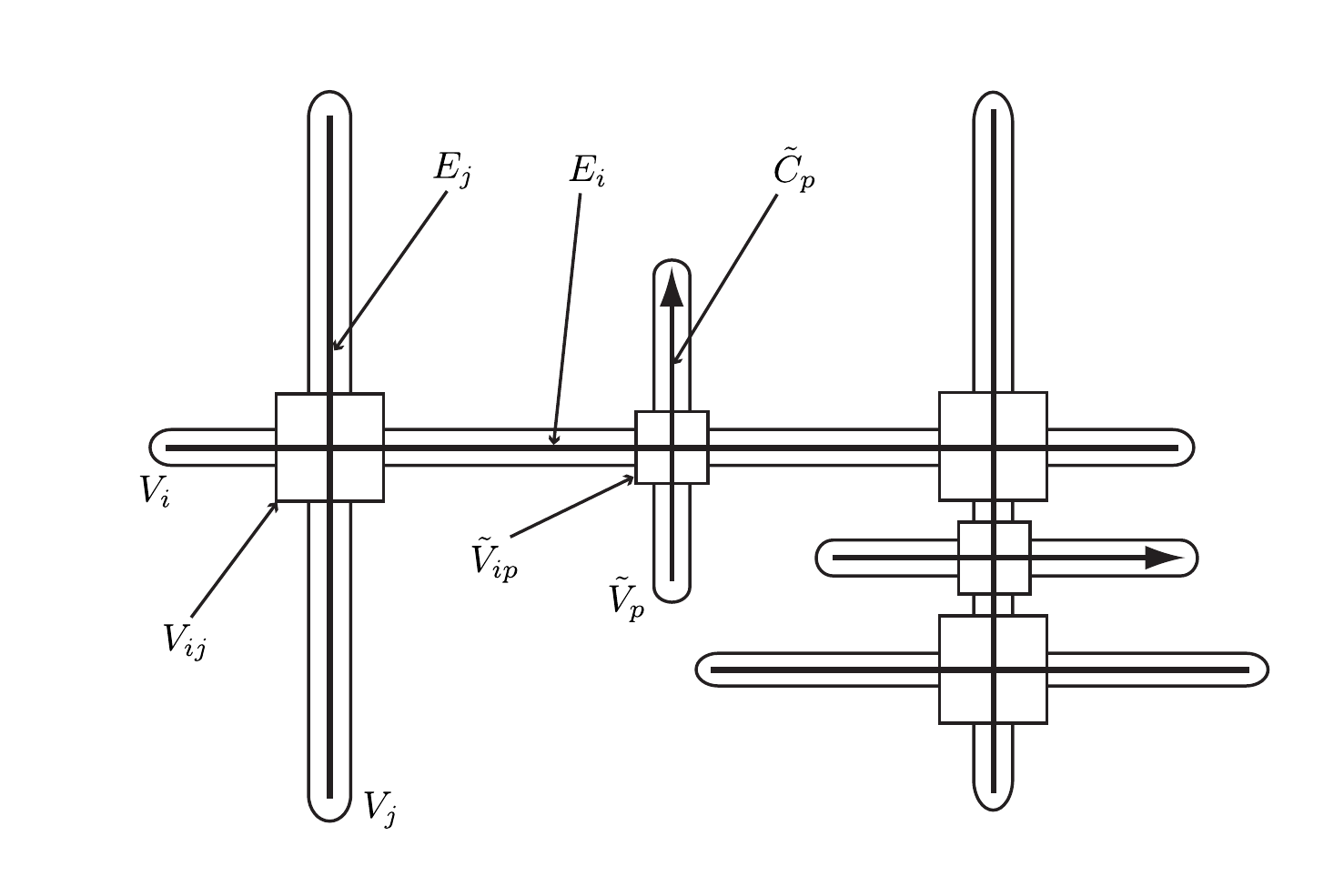}
\caption{}
\label{figT10}
\end{figure}

We decompose the Milnor fibre $F_\eta$ as follows:
$$F_\eta = \biggl( \bigcup_{i=1}^s \big( V_i \cap F_\eta \big) \biggr) \cup \biggl( \bigcup_{\substack{i,j =1 \\ i \neq j}}^s \big( V_{ij} \cap F_\eta \big) \biggr) \cup \biggl( \bigcup_{p=1}^w \big( \tilde{V}_p \cap F_\eta \big) \biggr) \cup \biggl( \bigcup_{\substack{1\leq i \leq s\\ 1 \leq p \leq w}} \big( \tilde{V}_{ip} \cap F_\eta \big) \biggr)$$

After making suitable coordinate changes, each part $V_{ij} \cap F_\eta$ (or $\tilde{V}_{ip} \cap F_\eta$, in a similar way) of the Milnor fibre $F_\eta$ has equation of the form 
$$x^{a_i} \bar{x}^{b_i} y^{a_j} \bar{y}^{b_j} \varphi_1 \overline{\varphi_2} = \eta$$
and each part $V_{i} \cap F_\eta$ (or $\tilde{V_p} \cap F_\eta$, in a similar way) of the Milnor fibre has local equation of the form 
$$x^{a_i} \bar{x}^{b_i} \varphi_1 \overline{\varphi_2} = \eta,$$
where $\varphi_1(x,y)$ and $\varphi_2(x,y)$ are units in $\bC\{x,y\}$, $a_i$ is the multiplicity of $f$ at $E_i$ and $b_i$ is the multiplicity of $g$ at $E_i$
(clearly either $a_p=1$ and $b_p=0$ or $a_p=0$ and $b_p=1$, for each $p= 1, \dots w$).

The pieces $V_i \cap F_\eta$ such that $a_i$ is equal to $b_i$ are more difficult to analyze. In order to make this task easier we will prove a preliminary result. Consider the dual graph of the
minimal embedded resolution of $(fg)^{-1}(0)$. Vertices are in bijection with irreducible components of the exceptional divisor. Two vertices are joined by an edge if and only if the associated exceptional divisors meet. 
We consider an arrowhead for each irreducible component of the strict transform of $(fg)^{-1}(0)$. We join each arrowhead with an edge to the vertex corresponding to the only irreducible component of 
the exceptional divisor that is met by the component of the strict transform represented by the arrowhead. The valency of a vertex is the number of incident edges. A vertex is a rupture vertex if its valency is at least
$3$.

\begin{lemma} \label{lemma0}
If $a_i$ is equal to $b_i$ then the valency of the vertex corresponding to $E_i$ equals $2$. Moreover no two components $E_i$ and $E_j$ with $a_i=b_i$ and $a_j=b_j$ have common intersection.
\end{lemma}
\begin{proof}
A. Pichon and J. Seade proved in \cite{PS} that $a_i=b_i$ implies that $E_i$ does not represent a rupture vertex of the dual graph of the embedded resolution. 

Now let $S$ be the union of all the exceptional divisors $E_i$ such that $a_i=b_i$ and consider its decomposition into connected components $S = S_1 \cup \dots \cup S_n$. 
For each $k$ we claim that all the vertices associated to divisors in $S_k$ have valency $2$. Suppose that there exists a component in $S_k$ whose vertex has valency equal to $1$.
Consider the subgraph containing the vertices associated to $S_k$ and the vertex corresponding to the unique $E_l \notin S_m$ such that $E_{i_1} \cap E_l \neq \emptyset$, for some $E_{i_1} \in S_m$ 
(note that $a_l \neq b_l$). Order the divisors of $S_k$ as $E_{i_1},...,E_{i_r}$ in such a way that $E_{i_j}$ meets $E_{i_{j+1}}$. 

If $M$ denotes the $\big(r \times (r+1) \big)$-intersection matrix of $S_k \cup E_l$ and $S_k$, that is,
$$
M =
\begin{pmatrix}
1 & E_{i_1}^2 & 1 & 0 & \cdots & 0 & 0\\
0 & 1 & E_{i_2}^2 & 1 & \cdots & 0 & 0\\
0 & 0 & 1 & E_{i_3}^2 & \cdots & 0 & 0\\
\vdots & \vdots & \vdots & \vdots & \ddots & \vdots &\vdots \\
0 & 0 & 0 & 0 & \cdots & E_{i_{r-1}}^2 & 1\\
0 & 0 & 0 & 0 & \cdots & 1 & E_{i_r}^2
\end{pmatrix}
,$$
and if $u_i$ denotes the number of intersection points between $E_i$ and the strict transform of $f$ minus the number of intersection points between $E_i$ and the strict transform of $g$, then we have
(see \cite{EN}, Theorem $18.2$):
$$M \cdot
\begin{pmatrix}
a_l-b_l\\
a_{i_1}-b_{i_1}\\
\vdots \\
a_{i_r}-b_{i_r}
\end{pmatrix}
+
\begin{pmatrix}
u_{i_1}\\
\vdots \\
u_{i_r}
\end{pmatrix}
=
\begin{pmatrix}
0\\ 
\vdots \\
0
\end{pmatrix}
$$
By assumption, the valency of $E_{i_j}$ equals two for $j<r$ and $1$ for $j=r$. This implies that $u_{i_j}=0$ for any $1\leq j\leq r$. Hence we have
$$
\begin{pmatrix}
1 & E_{i_1}^2 & 1 & \cdots & 0 & 0\\
\vdots & \vdots & \vdots & \ddots & \vdots &\vdots \\
0 & 0 & 0 & \cdots & E_{i_{r-1}}^2 & 1\\
0 & 0 & 0 & \cdots & 1 & E_{i_r}^2
\end{pmatrix}
\cdot
\begin{pmatrix}
a_l-b_l\\
0\\
\vdots \\
0\\
0
\end{pmatrix}
+
\begin{pmatrix}
0\\
\vdots \\
0
\end{pmatrix}
=
\begin{pmatrix}
0\\
\vdots \\
0
\end{pmatrix}
,$$
which implies that $a_l=b_l$, a contradiction. In fact, it follows that there cannot be two consecutive vertices with $a_i = b_i$,
that is, for each $m$ one has $S_m = E_{i_1}$. 
\end{proof}

The proof of the next Lemma involves some computations which are surprisingly involved if one compares with the holomorphic case. 
They appear because in our real analytic setting the local equations for $f\bar g$ involve units $\varphi_i$ which cannot be avoided by a coordinate change as in the holomorphic setting.

\begin{lemma} \label{lemma}
The intersection of the Milnor fibre $F_\eta$ with each neighbourhood $V_i$, $V_{ij}$, $\tilde{V_p}$ or $\tilde{V_{ip}}$ is homeomorphic to either a 
finite disjoint union of cylinders (cases $(i)$, $(iii)$ and $(iv)$ in the proof) or a finite covering over a multi-punctured disk (case $(ii)$ in the proof).
\end{lemma}

\begin{proof}
There are four cases to consider:
\\
\begin{itemize}

\item[$(i)$] The intersection $F_\eta \cap V_{ij}$ when $a_i \neq b_i$ and $a_j \neq b_j$ (and in a similar way for the intersection $F_\eta \cap \tilde{V}_{ip}$
when $a_i \neq b_i$): locally, we have 
$$(\fgbar \circ \pi) = x^{a_i} \bar{x}^{b_i} y^{a_j} \bar{y}^{b_j} \varphi_1 \overline{\varphi_2},$$
where 
$$
\begin{cases}
\varphi_1(x,y):= \a + \psi_1(x,y) \\
\varphi_2(x,y):= \b + \psi_2(x,y)
\end{cases} 
,$$
with $\alpha,\beta \in \bC \backslash \{0\}$ and $\psi_1(0) = \psi_2(0) =0$.
Setting
$$
\begin{cases}
\varphi_{1,t}(x,y):= \a + t \psi_1(x,y) \\
\varphi_{2,t}(x,y):= \b + t \psi_2(x,y)
\end{cases} 
,$$
we define the $1$-parameter family
$$h_t := f_t \overline{g_t} = x^{a_i} \bar{x}^{b_i} y^{a_j} \bar{y}^{b_j} \varphi_{1,t} \overline{\varphi_{2,t}},$$
defined on $V_{ij}$, which gives a homotopy between $h_1 = (\fgbar \circ \pi)$ and $h_0 = x^{a_i} \bar{x}^{b_i} y^{a_j} \bar{y}^{b_j} \a \bar{\b}$. 
We want to show that $h_1^{-1}(\eta)$ is homeomorphic to $h_0^{-1}(\eta)$.

Consider the real analytic mapping
$$
\begin{matrix}
H: & V_{ij} \times [0,1] & \to & \bC \times [0,1] \\
    & (z,t) & \mapsto & \big( h_t(z), t \big)
\end{matrix}
$$
and define
$$M := \big( V_{ij} \times [0,1] \big) \cap H^{-1} \big( \BD_\eta^* \times [0,1] \big).$$
Then consider the restriction
$$H|_M: M \to \BD_\eta^* \times [0,1].$$
We will show that $H|_M$ is the projection of a locally trivial fibration, and then $h_1^{-1}(\eta)$ is homeomorphic to the set
$$\left\{ x^{a_i} \bar{x}^{b_i} = \frac{\eta}{y^{a_j} \bar{y}^{b_j}}  \right\}\cap (D_{\epsilon_1}\times D_{\epsilon_2}),$$
for some small positive real $\e_1, \e_2 << \e$, which is a covering over an annulus, and therefore it is a disjoint union of cylinders.
\\
The proof that $H|_M$ is a locally trivial fibration follows from the Ehresmann's Fibration Lemma, since:

\begin{itemize}
\item[$\bullet$] $H|_M$ is clearly proper;
\item[$\bullet$] $H|_M$ is a submersion: the Jacobian matrix of $H$ is given by
$$
\begin{pmatrix}

\frac{\partial h_t}{\partial x} & \frac{\partial h_t}{\partial \bar{x}} & \frac{\partial h_t}{\partial y} & \frac{\partial h_t}{\partial \bar{y}} & \frac{\partial h_t}{\partial t}  \\

\frac{\partial \overline{h_t}}{\partial x} & \frac{\partial \overline{h_t}}{\partial \bar{x}} & \frac{\partial \overline{h_t}}{\partial y} & \frac{\partial \overline{h_t}}{\partial \bar{y}} & \frac{\partial \overline{h_t}}{\partial t} \\

\frac{\partial t}{\partial x} & \frac{\partial t}{\partial \bar{x}} & \frac{\partial t}{\partial y} & \frac{\partial t}{\partial \bar{y}} & \frac{\partial t}{\partial t}

\end{pmatrix}
= 
\begin{pmatrix}

\frac{\partial f_t}{\partial x} \overline{g_t} & f_t \overline{\frac{\partial g_t}{\partial x}} & \frac{\partial f_t}{\partial y} \overline{g_t} & f_t \overline{\frac{\partial g_t}{\partial y}} & x^{a_i} y^{a_j} \psi_1 \bar{g_t} \\

\bar{f_t} \frac{\partial g_t}{\partial x} & \overline{\frac{\partial f_t}{\partial x}} g_t & \bar{f_t} \frac{\partial g_t}{\partial y} & \overline{\frac{\partial f_t}{\partial y}}g_t & x^{b_i} y^{b_j} \psi_2 \bar{f_t} \\

0 & 0 & 0 & 0 & 1 \\

\end{pmatrix}
$$ 

Then $H|_M$ is a submersion in a point $p$ if, and only if, $p$ is not a solution of at least one of the following equations:
$$
\begin{cases}
(1) & |\frac{\partial f_t}{\partial x}|^2 |g_t|^2 -  |\frac{\partial g_t}{\partial x}|^2 |f_t|^2 =0  \\
(2) & |\frac{\partial f_t}{\partial y}|^2 |g_t|^2 -  |\frac{\partial g_t}{\partial y}|^2 |f_t|^2 =0  \\
(3) & |f_t|^2 \frac{\partial g_t}{\partial x} \overline{\frac{\partial g_t}{\partial y}} - |g_t|^2 \frac{\partial f_t}{\partial x} \overline{\frac{\partial f_t}{\partial y}} =0 \\ 
(4) & fg \big( \frac{\partial f_t}{\partial x}\frac{\partial g_t}{\partial y} - \frac{\partial f_t}{\partial y}\frac{\partial g_t}{\partial x} \big) =0
\end{cases}
$$

Note that $f_t= x^{a_i} y^{a_j} \varphi_{1,t}$ and $g_t= x^{b_i} y^{b_j} \varphi_{2,t}$. Then setting
$$
\begin{cases}
\zeta_1 = a_i \varphi_{1,t} + x \frac{\partial \varphi_{1,t}}{\partial x} \\
\zeta_2 = a_j \varphi_{1,t} + y \frac{\partial \varphi_{1,t}}{\partial y} \\
\zeta_3 = b_i \varphi_{2,t} + x \frac{\partial \varphi_{2,t}}{\partial x} \\
\zeta_4 = b_j \varphi_{2,t} + y \frac{\partial \varphi_{2,t}}{\partial y}
\end{cases}
$$
we have that
$$
\begin{cases}
\frac{\partial f_t}{\partial x} = x^{a_i-1} y^{a_j} \zeta_1 \\
\frac{\partial f_t}{\partial y} = x^{a_i} y^{a_j-1} \zeta_2 \\
\frac{\partial g_t}{\partial x} = x^{b_i-1} y^{b_j} \zeta_3 \\
\frac{\partial g_t}{\partial y} = x^{b_i} y^{b_j-1} \zeta_4 \\
\end{cases}
$$

Substituting on equations $(1)$ to $(4)$, we have the equations:
$$
\begin{cases}
(1) & |x|^{2(a_i+b_i-1)} |y|^{2(a_j+b_j)}
\begin{vmatrix}
|\zeta_1|^2 & |\zeta_3|^2 \\
|\varphi_{1,t}|^2 & |\varphi_{2,t}|^2
\end{vmatrix}
=0 \\
  
(2) & |x|^{2(a_i+b_i)} |y|^{2(a_j+b_j-1)}
\begin{vmatrix}
|\zeta_2|^2 & |\zeta_4|^2 \\
|\varphi_{1,t}|^2 & |\varphi_{2,t}|^2
\end{vmatrix}
=0 \\

(3) & \bar{x} y |x|^{2(a_i+b_i-1)} |y|^{2(a_j+b_j-1)}
\begin{vmatrix}
\zeta_1 \overline{\zeta_2} & \zeta_3 \overline{\zeta_4} \\
|\varphi_{1,t}|^2 & |\varphi_{2,t}|^2
\end{vmatrix}
=0 \\

(4) & x^{2a_i+2b_i-1} y^{2a_j+2b_j-1} \varphi_{1,t} \varphi_{2,t}
\begin{vmatrix}
\zeta_1 & \zeta_2 \\
\zeta_3 & \zeta_4
\end{vmatrix}
=0

\end{cases}
$$

Since
$$ 
\begin{vmatrix}
|\zeta_1(0)|^2 & |\zeta_3(0)|^2 \\
|\varphi_{1,t}(0)|^2 & |\varphi_{2,t}(0)|^2
\end{vmatrix}
= |a_i^2-b_i^2|.|\varphi_{1,t}(0) \varphi_{2,t}(0)|^2
$$

and
$$ 
\begin{vmatrix}
|\zeta_2(0)|^2 & |\zeta_4(0)|^2 \\
|\varphi_{1,t}(0)|^2 & |\varphi_{2,t}(0)|^2
\end{vmatrix}
= |a_j^2-b_j^2|.|\varphi_{1,t}(0) \varphi_{2,t}(0)|^2,
$$
and since by hypothesis $a_i \neq b_i$ and $a_j \neq b_j$, it follows that if $p$ is sufficiently close to $0$, then $p$ cannot be a solution of equations $(1)$ and $(2)$.

\item[$\bullet$] The restriction of $H$ to the boundary of $M$ is a submersion.
Indeed, the tubular neighbourhoods can be chosen so that $V_{ij}$ is the closed polydisk $D_{\epsilon_1}\times D_{\epsilon_2}$ for small $\epsilon_i$.
Note that $\delta<<min\{\epsilon_1,\epsilon_2\}$.
Then $\partial ( \BD_{\e_1} \times \BD_{\e_2}) = (\partial \BD_{\e_1} \times \BD_{\e_2}) \cup (\BD_{\e_1} \times \partial \BD_{\e_2})$ and $h_t^{-1}(\delta)$ intersects $\partial \BD_{\e_1} \times \BD_{\e_2}$ transversally if, and only if, the Jacobian matrix
$$
\begin{pmatrix}

\frac{\partial h_t}{\partial x} & \frac{\partial h_t}{\partial \bar{x}} & \frac{\partial h_t}{\partial y} & \frac{\partial h_t}{\partial \bar{y}}   \\

\frac{\partial \overline{h_t}}{\partial x} & \frac{\partial \overline{h_t}}{\partial \bar{x}} & \frac{\partial \overline{h_t}}{\partial y} & \frac{\partial \overline{h_t}}{\partial \bar{y}}  \\

\frac{\partial x}{\partial x} & \frac{\partial x}{\partial \bar{x}} & \frac{\partial x}{\partial y} & \frac{\partial x}{\partial \bar{y}}  \\

\frac{\partial \bar{x}}{\partial x} & \frac{\partial \bar{x}}{\partial \bar{x}} & \frac{\partial \bar{x}}{\partial y} & \frac{\partial \bar{x}}{\partial \bar{y}}

\end{pmatrix}
$$
has non-zero determinant at $p \neq 0$ small, which happens if equation $(1)$ holds, and therefore if $a_i \neq b_i$. So we conclude that 
$h_t^{-1}(\delta)$ intersects $\partial \BD_{\e_1} \times \BD_{\e_2}$ transversally and, by an analogous argument,
that $h_t^{-1}(\delta)$ intersects $\BD_{\e_1} \times \partial \BD_{\e_2}$ transversally if $a_j \neq b_j$.
\end{itemize}

\item[$(ii)$] The intersection $F_\eta \cap V_i$ when $a_i \neq b_i$ (and in a similar way for the intersection $F_\eta \cap \tilde{V_p}$):
the same computations of case $(i)$, considering $(\fgbar \circ \pi) = x^{a_i} \bar{x}^{b_i} \varphi_1 \overline{\varphi_2}$ show 
that for any point $q\in V_i\cap E_i$ there exists a small neighbourhood $U_q$ of $q$ and a smooth coordinate system $(x,y):U_q\to\mathbb{C}^2$ such that $\fgbar$ is given by 
$\fgbar(x,y)=x^{a_i}\bar{x}^{b_i}\alpha\bar\beta$ with $\alpha,\beta\in\mathbb{C}^*$. Hence for any sequence $\{q_n\}_{n\in\mathbb{N}}$ converging to $q$ we have that the limit of tangent 
planes $T_{q_n}(\fgbar)^{-1}(\fgbar(q_n))$ converges to $T_qE_i$. Therefore if $\pi:V_i\to V_i\cap E_i$ is any smooth projection, for sufficiently small $\eta$ we have that
$$\pi|_{(\fgbar)^{-1}(\eta)\cap V_i}: (\fgbar)^{-1}(\eta)\cap V_i\to E_i$$
is an etale covering of degree $|a_i-b_i|$. Notice that $V_i \cap E_i$ is a multi-punctured disk.

\item[$(iii)$] The intersection $F_\eta \cap V_{ij}$ when $a_i =b_i$ and $a_j \neq b_j$ or when $a_i \neq b_i$ and $a_j = b_j$
(and in a similar way the intersection $F_\eta \cap \tilde{V_{ip}}$ when $a_i=b_i$): consider the map-germ 
$$
\begin{matrix}
(f \circ \pi , g \circ \pi): & (\bC^2,0) & \to & (\bC^2,0) \\
          & (x,y)   & \mapsto & (x^{a_i} y^{a_j} \varphi_1, x^{b_i} y^{b_j} \varphi_2)
\end{matrix}
$$
We want to find a change of coordinates $\Theta = (\Theta_1, \Theta_2): (\bC^2,0) \to (\bC^2,0)$ such that
$$(f \circ \pi ,g \circ \pi) = (x^{a_i} y^{a_j}, x^{b_i} y^{b_j}) \circ \Theta,$$
which happens if and only if $(f \circ \pi ,g \circ \pi) = (\Theta_1^{a_i} \Theta_2^{a_j}, \Theta_1^{b_i} \Theta_2^{b_j})$.
If we set $\Theta_1 = x \theta_1$ and $\Theta_2 = y \theta_2$, with $\theta_1(0) \neq 0$ and $\theta_2(0) \neq 0$, our problem is to find $\theta_1, \theta_2: (\bC^2,0) \to (\bC,0)$ such that
$$(f \circ \pi ,g \circ \pi) = (x^{a_i} y^{a_j} \theta_1^{a_i} \theta_2^{a_j}, x^{b_i} y^{b_j} \theta_1^{b_i} \theta_2^{b_j}).$$
This is possible if, and only if, the system
$$
\begin{cases}
x^{a_i} y^{a_j} \theta_1^{a_i} \theta_2^{a_j} = x^{a_i} y^{a_j} \varphi_1 \\
x^{b_i} y^{b_j} \theta_1^{b_i} \theta_2^{b_j} = x^{b_i} y^{b_j} \varphi_2
\end{cases}
$$
has a solution $(\theta_1,\theta_2)$. This is equivalent to the system
$$
\begin{cases}
\theta_1^{a_i} \theta_2^{a_j} = \varphi_1 \\
\theta_1^{b_i} \theta_2^{b_j} = \varphi_2,
\end{cases}
$$
which has solution if, and only if, the linear system with variables $log\theta_1,log\theta_2$
$$
\begin{cases}
a_i \log \theta_1 + a_j \log \theta_2 = \log \varphi_1 \\
b_i \log \theta_1 + b_j \log \theta_2 = \log \varphi_2
\end{cases}
$$
has solution. But by hypothesis we have that the determinant
$$
\begin{vmatrix}
a_i & a_j \\
b_i & b_j
\end{vmatrix}
\neq 0,
$$
and hence the system has solutions.
Then $(\fgbar \circ \pi) = (x^{a_i} \bar{x}^{b_i} |y|^{2b_i}) \circ \Theta$ and therefore its Milnor fibre is given by the set
$$\{|y|^{2b_j} = \frac{\eta}{x^{a_i} \bar{x}^{b_i}}\}\cap V_{ij},$$
which is a disjoint union of cylinders (which intersect $\partial V_{ij}$ transversely).

\item[$(iv)$] The intersection $F_\eta \cap V_i$ when $a_i = b_i$ (and in a similar way the intersection $F_\eta \cap V_{ij}$ when $a_i=b_i$ and $a_j=b_j$).
Lemma~\ref{lemma0} states that the components $E_{k}$ with $a_k = b_k$ do not meet each other and their associated vertex in the dual graph has valency equal to $2$. 

Set $\BS_\eta := \partial \BD_\eta$, the boundary of the disk $\BD_\eta$. Looking at the resolution $\pi$, it is not difficult to see that the piece 
$$(\fgbar \circ \pi)^{-1}(\BS_\eta)\cap V_k$$
of the Milnor tube $(\fgbar \circ \pi)^{-1}(\BS_\eta)$ has the homotopy type of a torus $\bT^2$. Indeed,
$$(\fgbar \circ \pi)^{-1}(D^*_\eta)\cap V_k$$
is diffeomorphic to $D\times D^*$, where $D$ is a disk.

Since the restriction 
$$(\fgbar \circ \pi)_|: (\fgbar \circ \pi)^{-1}(\BS_\eta) \cap V_k \to \BS_\eta$$
is the projection of a locally trivial fibration, and supposing that $F_\eta \cap V_k$ is connected, we get an exact sequence
$$\pi_2(\BS^1) \to \pi_1(F_\eta \cap V_k) \to \pi_1(\bT^2) \to \pi_1(\BS^1) \to 0,$$
which is isomorphic to
$$0 \to \pi_1(F_\eta \cap V_k) \to \bZ \times \bZ \to \bZ \to 0.$$
Hence $\pi_1(F_\eta \cap V_k)$ is isomorphic to $\bZ$, and since $F_\eta \cap V_k$ has two boundary components (two circles), it follows that $F_\eta \cap V_k$ is a cylinder.\\
If the piece of the Milnor fibre $F_\eta \cap V_k$ has more than one connected component, since the part of the Milnor tube $(\fgbar \circ \pi)^{-1}(\BS_\eta) \cap V_k$ is connected, the monodromy $m$ must define a one-cycle permutation $(F_1, \dots, F_r)$ on the connected components of $F_\eta \cap V_k$. 
In particular, they are all diffeomorphic. For $i \neq r$, let $m_i$ denote the diffeomorphism from $F_i$ to $F_{i+1}$, and let $m_r$ be the diffeomorphism from $F_r$ to $F_1$. Then we can construct a fibre bundle over the circle with total space 
$$(\fgbar \circ \pi)^{-1}(\BS_\eta) \cap V_k$$ 
and with fibre $F_1$ in the following way: consider the circle given by the identification of the boundary of the interval $[0,r]$. The space $(\fgbar \circ \pi)^{-1}(\BS_\eta) \cap V_k$ is diffeomorphic to the quotient of $\coprod_{i=0}^{r-1}[i,i+1]\times F_{i-1}$ by the identification of $(i,x)$ with $(i+1,m_i(x))$ for any $i<r$ and $(r,x)$ with $(0,m_r(x))$. The projection to the circle is now obvious. Then the same arguments above show that $F_1$ is a cylinder.
\end{itemize}
\end{proof}

We will need the following Remark on families of maps of the form $f_t\bar{g}_t$, which is an adaptation of the corresponding result for families of holomorphic functions.

\begin{remark}
\label{lemma_2}
Let $f_t:(X,0)\to\bC$ and $g_t:(X,0)\to\bC$ be two families of holomorphic function-germs defined on a germ of complex surface $X$ and depending holomorphically on a parameter $t$ which varies in a disk $D$.
Suppose that the zero set of each $f_t g_t$ has an isolated singularity at the origin, whose Milnor number is independent of $t$. 
For any $t_0\in D$, there exists a neighbourhood $U$ of $t_0$ in $D$ and a positive real number $\epsilon$ which is a Milnor radius for $f_t\bar{g}_t$, for any $t\in U$.
\end{remark}

\begin{proof}
A Milnor radius for a real analytic germ of the form $f\bar{g}$ is a positive radius $\epsilon$ such that for any other positive radius $\epsilon'\leq\epsilon$, the sphere of radius $\epsilon'$ centered at the origin
 meets $(fg)^{-1}(0)$ transversely. Thus $\epsilon$ is a Milnor radius for the real analytic germ $f\bar{g}$ if and only if it is a Milnor radius for the holomorphic germ $fg$. The assertion of the Remark follows from the corresponding one for $\mu$-constant families of holomorphic function-germs in two variables: its well known that a $\mu$-constant family of plane curves is Whitney equisingular and admits an uniform Milnor radius.
\end{proof}

\section{The Milnor fibration in a tube}
\label{section_3}

Let $(X,0)$ be a $3$-dimensional isolated singularity complex analytic germ at the origin. Let $f$ and $g$ be holomorphic function-germs (not necessarily reduced) defined in $(X,0)$ such that 
$V(f):=f^{-1}(0)$ and $V(g):=g^{-1}(0)$ have no common irreducible components and such that the real analytic map-germ $\fgbar: (X,0) \to (\bC,0)$ has an isolated critical value. Let $\Sigma$ denote the singular set
of the surface $V(fg)$, considered with reduced structure. 

Consider the polydisk 
$$\Delta_{\underline\e}:= \{ (z_1, \dots, z_N) \in \bC^N: |z_1|\leq\e_1, \dots, |z_N|\leq \e_N \}.$$
By choosing the coordinate system and the radii $\epsilon_i$ we may assume that $\partial \Delta_{\underline\e}$ is the union of the faces $|z_i|=\e_i$ for $i=1,2,3$ and that $\Sigma$ meets $\partial \Delta_{\underline\e}$
at the interior of the face $|z_1|=\e_1$ in a transverse way.

In the case that $X$ has complex dimension two, Pichon and Seade proved in \cite{PS} that the real analytic map-germ $\fgbar$ has the Thom $a_f$ condition if and only if it has an isolated 
critical value. Moreover they showed the following numerical criterion at the level of the embedded resolution of the pair $(X,(fg)^{-1}(0))$: a rupture irreducible component of the
total transform of $(fg)^{-1}(0)$ to its embedded resolution of singularities is an irreducible component of the total transform that meets other irreducible components in at least three 
points; $\fgbar$ has the Thom $a_f$ conditions if and only if the multiplicity of $f$ is different to the multiplicity of $g$ at any rupture irreducible component of the total transform.

Thom's $a_f$ condition implies that $\fgbar$ admits a Milnor fibration
in the tube. That is, there exist 
sufficient small positive reals $0 < \eta << \e_i<<1$ such that the restriction 
$$\fgbar_{| (\fgbar)^{-1}(\BD_\eta^*) \cap X \cap \Delta_{\underline\e}}: (\fgbar)^{-1}(\BD_\eta^*) \cap X \cap \Delta_{\underline\e} \to \BD_\eta^*$$
is a locally trivial fibration.

In higher dimensions this result is no longer true, as one can check in the following example, pointed out by Parusinski: take 
$f(z_1,z_2,z_3)= z_1(z_2 + z_3^2)$ and $g(z_1,z_2,z_3)= z_2$. The function $\fgbar$ has an isolated critical value and does not admit a Milnor fibration in the tube. 
We have the following theorem, which gives necessary and sufficient conditions for $\fgbar: (X,0) \to (\bC,0)$ to have a Milnor fibration in the tube in the case that $X$ is of
dimension $3$:

\vskip.2cm
\noindent
{\bf Theorem.} {\it 
For each irreducible component $\Sigma_i$ of the intersection $V(f) \cap V(g)$, let $H_i \subset X$ be a generic transversal section, with $i= 1, \dots, k$. 
Then $\fgbar$ has a Milnor fibration in the tube if and only if 
the restriction of $\fgbar$ to each $H_i$ has isolated critical values (notice that as explained above this is equivalent to an easy to check numerical criterion at the level of the 
embedded resolution of the restriction of $fg$ to $H_i$). 
}

\begin{proof}
We known by Ehresmann Fibration Theorem that $\fgbar$ has a Milnor fibration in the tube if and only if the smooth manifold 
$(\fgbar)^{-1}(t)$ intersects $\partial \Delta_{\underline\e} \cap X$ transversely, for any $t \in \bC$ small. 

At a point $q \in (fg)^{-1}(0) \cap \partial\Delta_{\underline\e}$ which does not belong to $\Sigma$ the function $\fgbar$ has the form either $w^a\bar{u}$ or $\bar{w}^au$, where $u$ is a 
holomorphic function non-vanishing at $q$, $a$ is a positive integer and $w$ is a holomorphic function vanishing and non-critical at $q$. If the choice of the coordinates $z_1,...,z_n$
is generic enough
we may assume that if $q$ belongs to the face $z_i=\e_i$ ($i=1,2,3$) then $w,z_i$ are part of a local coordinate system in $X$ at $q$. 
Similarly if $q$ belongs to the intersection
of the faces $z_i=\e_i$ and $z_j=\e_j$ ($i,j=1,2,3$) then $w,z_i,z_j$ form a local coordinate system of $X$ at $q$. The local form of $\fgbar$ implies immediately that if $\{q_n\}_{n\in\mathbb{N}}$ is any 
sequence outside $(fg)^{-1}(0)$ converging to $q$ we have that $T_{q_n}\big( (\fgbar)^{-1}(\fgbar|_X(q_n) \big)$ converges to $T_q \big( (\fgbar|_X)^{-1}(0) \big)$.  
Taking into account that $w$ forms a coordinate system together with $z_i$ ($i=1,2$), this implies that outside a fixed neighbourhood $U$ of $\Sigma\cap\partial\Delta_{\underline\e}$ in 
$\partial\Delta_{\underline\e}$, if $t\neq 0$ is small enough then $(\fgbar)^{-1}(t)$ intersects $\partial \Delta_{\underline\e} \cap X$ transversely. 
We record, for later use, the following second implication of the property verified above for the limits of tangent planes:\\

($\dagger$) Let $U$ be a fixed tubular neighbourhood of $\Sigma\cap\partial\Delta_{\underline\e}$ in $\partial\Delta_{\underline\e}$. Since 
$(fg)^{-1}(0)\cap \partial\Delta_{\underline\e}\setminus U$ is a smooth submanifold of real codimension $2$ of $\partial\Delta_{\underline\e}\setminus U$ there exists a tubular 
neighbourhood $W$ of $(fg)^{-1}(0)\cap \partial\Delta_{\underline\e}\setminus U$ in $\partial\Delta_{\underline\e}\setminus U$, together with a projection 
$$\rho:W\to (fg)^{-1}(0)\cap \partial\Delta_{\underline\e}\setminus U$$
which is a locally trivial fibration with fibre a disc. For a fixed small enough $\eta\neq 0$ the restriction 
$$\rho|_{(\fgbar)^{-1}(\eta)\cap W}: (\fgbar)^{-1}(\eta)\cap W\to (fg)^{-1}(0)\cap \partial\Delta_{\underline\e}\setminus U$$
is a unramified covering.\\

To finish the proof notice that $\Sigma$ meets $\partial\Delta_{\underline\e}$ at the interior of the face $z_1=\e_1$, and that at any point $q\in \Sigma\cap \partial\Delta_{\underline\e}$
the hyperplane $\{z_1=z_1(q)\}$ meets $\Sigma$ transversely. Thus, transversality of $(\fgbar)^{-1}(\eta)$ with $\partial\Delta_{\underline\e}$ 
for $\eta$ small enough near $\Sigma$ is equivalent to the fact that the 
restriction of $\fgbar$ to a neighbourhood of $q$ in $\{z_1=z_1(q)\}$ has an isolated critical value.
\end{proof}

Then we have proved $(i)$ of Theorem \ref{theo_main}.

In the example above, we have $\Sigma_1 = \{z_2 = z_3 =0 \}$ and $H_1= \{z_1=p \}$, for some $p \in \bC$ small. Then $\fgbar_{|H_1} (z_2,z_3) = (z_2 + z_3^2) \bar{z_2}$, which has no 
isolated critical value.

\section{The boundary of the Milnor fibre of non-isolated singularities}
\label{section_4}

We continue with the setting of previous section. For each irreducible component $\Sigma_i$ of the intersection $V(f) \cap V(g)$, let $H_i$ be a generic transversal section. Suppose that each $H_i$ has an isolated critical value. 
Then $\fgbar$ has a Milnor fibration in the tube, as we have proved.

Our goal is to prove that the boundary of the Milnor fibre of $\fgbar: (X,0) \to (\bC,0)$ is a Waldhausen manifold, that is, we will prove that the set
$$L_t := (\fgbar)^{-1}(t) \cap X \cap \partial \Delta_{\underline\e}$$
is a Waldhausen manifold, for any $t \in \BD_\eta^*$, where $\Delta_{\underline\e} \subset \bC^N$ is the polydisk 
for $\underline{\e}=(\e_1,\dots,\e_N)$, and $\partial \Delta_{\underline\e}$ is its boundary.

Consider 
$$L_0 := (\fgbar)^{-1}(0) \cap X \cap \partial \Delta_{\underline\e} = (fg)^{-1}(0) \cap X \cap \partial \Delta_{\underline\e}$$
with its reduced structure. Its singular locus is the intersection of $\partial \Delta_{\underline\e}$ with the complex curve $\Sigma$ given by the singular locus of $(fg)^{-1}(0)$, 
that is, $L(\Sigma):= \Sigma \cap \partial \Delta_{\underline\e}$.
After an appropriate choice of $\underline\e$ we have that $L(\Sigma)$ is a finite disjoint union of circles $\BS^1$ contained in the open face
$$\mathcal{S}:= \{(z_1,\dots,z_N)\in\bC^N: |z_1|=\epsilon_1,\dots,|z_{N-1}|<\epsilon_{N-1}, |z_N|<\epsilon_N\}.$$

If $\Sigma = \Sigma_1 \cup \dots \cup \Sigma_k$ is the decomposition of $\Sigma$ into irreducible components, we define the following tubular neighbourhood of
$L(\Sigma)$ in $\partial \Delta_{\underline\e}$:
$$W = W[1] \sqcup \dots \sqcup W[k],$$
where $W[l]$ is a small tubular neighbourhood of the circle $L(\Sigma_l):= \Sigma_l \cap \partial \Delta_{\underline\e}$ in 
$\mathcal{S} \subset \partial \Delta_{\underline\e}$, for $l = 1, \dots, k$. 

Given a point $p\in\Sigma_l\backslash\{0\}$, let $H_p$ be the $2$-dimensional affine hyperplane of $X$ passing through $p$ and parallel to the hyperplane $\{z_1=0\}$. 
Choosing $\underline\epsilon$ small enough we may assume that the Milnor number of the function-germ $(fg|_{H_p},p)$ is independent of $p$. 

Therefore, by Remark \ref{lemma_2} and the compactness of $L(\Sigma_l)$, we deduce the existence of a positive real number $\delta$ such that for any 
$p \in L(\Sigma_l)$ the ball in $H_p$ centered at $p$ and of radius $\delta$ is a Milnor ball for the function-germ $(fg|_{H_p},p)$. We may choose $W[l]$ to be the union of
those balls when $p$ varies in $L(\Sigma_l)$. With this definition, there is a natural fibration
\begin{equation}
\label{fibracionl}
\sigma_l:W[l] \to L(\Sigma_l)
\end{equation}
with fibre a complex $2$-ball.

Since $f\bar g$ satisfies the Thom's $a_f$ condition, there exists a positive $\eta$ such that the mapping
$$\Psi_l:W[l]\cap (f\bar g)^{-1}(D_\eta) \to D_\eta \times L(\Sigma_l)$$
defined by $\Psi_l:=(f\bar g,\sigma_l)$ has only the circle $\{0\}\times L(\Sigma_l)$ as critical values. Therefore, for $t \in D_\eta^*$, the restriction
$$\sigma_l|_{W_t[l]}: W_t[l] \to L(\Sigma_l)$$
is a locally trivial fibration, where $W_t[l]:= L_t \cap W[l]$. Its fibre is the so-called {\it transversal Milnor fibre of $f\bar g$ at $\Sigma_l$} and its monodromy $h$ is the 
{\it vertical monodromy along $\Sigma_l$} (see~\cite{Si}).

\subsection{The exterior part}
By the assertion $(\dagger)$ we can take $\delta$ (and hence $W$) sufficiently small such that $L_t \backslash \mathring{W}$ is a unramified covering of $L_0 \backslash \mathring{W}$,
where $\mathring{W}$ is the interior of $W$ (see \cite{Si} for instance). In this case, $W$ is said to be a {\it vanishing zone for $\fgbar$}.
Set $W_0:= W \cap L_0$ and $W_t:= W \cap L_t$.

Let 
$$n: \tilde{Y} \to (fg)^{-1}(0)$$
be the normalization of $(fg)^{-1}(0)$. Set $\tilde{\Sigma}:= n^{-1}(\Sigma)$ and $\tilde{L}_0 := n^{-1}(L_0)$.

It is easy to see that $\tilde{Y}\backslash \tilde{W}_0$ is a Waldhausen manifold, where $\tilde{W}_0:= n^{-1}(W_0)$ is a tubular neighbourhood
of $\tilde{\Sigma}$ in $\tilde{X}$ (see \cite{MP} for instance). Since the normalization is a biholomorphism outside $\tilde{W}_0$, we find that $L_0 \backslash W_0$
is a Waldhausen manifold. Then $L_t\setminus W_t$ is a Waldhausen manifold as well due to the following results: 

\begin{prop} \label{prop_wald}
A regular finite covering of a Seifert manifold is a Seifert manifold.
\end{prop}

\begin{proof} 
Let $\pi: M' \to M$ be a finite covering of a Seifert manifold $M$, which is a fibre bundle over a compact surface with boundary $F$ and fibre $\BS^1$ (with finitely many special fibres), and projection $p: M \to F$. The composition
$$p \circ \pi: M' \to F$$
is a fibre bundle (with finitely many special fibres) with fibre a finite covering of $\BS^1$. If such fibre is connected, then we have expressed $M'$ as a bundle over a surface with fibre $\BS^1$, and therefore it is a Seifert manifold. If the fibre is not connected, we define the following equivalence relation in $M'$: two points are identified if they belong to the same connected component of the same fibre of $p \circ \pi$. The quotient of $M'$ by this equivalence relation is a surface $F'$ that covers $F$, and the quotient application $q: M' \to F'$ expresses $M'$ as a bundle over $F'$ with fibre $\BS^1$.
\end{proof}

\begin{cor}\label{cor_wald}
A regular finite covering of a Waldhausen manifold is a \-Waldhausen manifold.  
\end{cor}

\begin{proof}
 The finite cover is decomposed in regular finite covers of Seifert manifolds cutting along tori, according to the decomposition of the base.
\end{proof}

\subsection{The vanishing zone}

Now, since $\partial W_t$ is a finite disjoint union of tori, all we have to prove is that $W_t$ is a Waldhausen manifold. To prove that, we have to show that each
connected component $W_t[l]$, for $l= 1, \dots, k$, is a Waldhausen manifold. By the definition of Waldhausen manifold, it is sufficient to give a decomposition of
each transversal Milnor fibre which is invariant under the corresponding vertical monodromy $h_l$ and such that the corresponding pieces of $W_t[l]$ are Waldhausen manifolds.
We will prove that these pieces are Waldhausen manifolds either by proving that they are fibre bundles over a circle with fibre a cylinder, or by showing directly that they are 
finite unramified coverings of Walhausen manifolds.

Let us fix an irreducible component $\Sigma_l$ of $\Sigma$. 
Recall the fibration~(\ref{fibracionl}) has fibre $W[l]_s$ over $L(\Sigma_l)$ a complex $2$-ball centered in $s$ and contained in the smooth complex surface $X\cap\{z_1=z_1(s)\}$. 
We can look at the restriction $(\fgbar)|_{W[l]}:W[l] \to \bC$ as a family in the parameter $\BS^1$. For each $s \in L(\Sigma_l)$, we denote by $f_s \bar{g}_s$ the restriction of 
$\fgbar$ to $W[l]_s$. We have have taken $W[l]$ so that each $W[l]_s$ is a Milnor ball for $f_s \bar{g}_s$.

It is well known that then we can consider a minimal embedded resolution in family 
$$\pi: \tilde{M} \to W[l],$$ 
which is an analytic morphism $\pi$ where:

\begin{itemize}
\item[$\bullet$] $A:=((\fgbar)|_{W[l]}\circ\pi)^{-1}(0)$ is the total transform of the zero set of $fg$ in $W[l]$, with a decomposition into irreducible components 
$A = \cup_{i=1}^r A_i$, where an irreducible component is defined as the closure of a connected component of $A \backslash Sing(A)$;
\item[$\bullet$] for each $s \in L(\Sigma_s)$, define $X_s := \pi^{-1}(W[l]_s)$. Then 
$$\pi_s: X_s \to B_\delta$$
is the minimal embedded resolution of the plane curve singularity defined by the restriction of $fg$ to $H_s$. Denote by $A^s$ the divisor $A\cap X_s$ and set $A^s_i:=A_i \cap X_s$. Notice that $A^s_i$ may consist of several irreducible components of $A^s$. 
\end{itemize}

Note that $\sigma_l \circ \pi: \tilde{M} \to L(\Sigma_l)$ is a fibre bundle with fibre $X_s$.

Recall the regions $V_i$, $V_{ij}$, $\tilde{V}_p$ and $\tilde{V}_{ip}$ introduced in Section~\ref{section_2}. In order to simplify notation, from now on, we shall make no distinction 
between boxes of the type 
$V_i$ and $\tilde{V}_p$ nor between boxes of the type $V_{ij}$ and $\tilde{V}_{ip}$.

For each $s \in L(\Sigma_l)$, define $V_i^s$ to be the union of the boxes $V_j$ introduced in Section~\ref{section_2} corresponding to irreducible components of the divisor $A^s$ which are contained in $A^s_i$. 
Analogously, define $V_{ij}^s$ in $X_s$ as the union of the boxes $V_{kl}$ introduced in Section~\ref{section_2} corresponding to intersections of two irreducible components of $A^s$, one contained in 
$A^s_k$, and the other in $A^s_l$ (notice that since $A^s_i$ may consist of several irreducible components, the region $V_{ii}^s$ may be non-empty). 

From now on, define $T_i:=\cup_{s\in L(\Sigma_l)}V_i^s$ and  $T_{ij}:=\cup_{s\in L(\Sigma_l)}V_{ij}^s$. Note that each $T_i$ (resp. $T_{ij}$) is a fibre bundle over $L(\Sigma_l)$ with fibre $V_i^s$ 
(resp. $V_{ij}^s$).

Clearly, the transversal Milnor fibre $(f\bar g)^{-1}(t)\cap \sigma_l^{-1}(s)$ is diffeomorphic to 
$$F_t(s):= (f_s \bar{g}_s \circ \pi_s)^{-1}(t),$$
and the vertical monodromy of $\fgbar$ can be taken to be the automorphism $\tilde{h}:= \tilde{H}|_{F_t(s)}: F_t(s) \to F_t(s)$. Moreover, $F_t(s)$ can be decomposed in 
$\tilde{h}$-invariant subsets as follows:
$$F_t(s) = \biggl( \bigcup_i \big( F_t(s) \cap V_i^s \big) \biggr) \cup \biggl( \bigcup_{i,j} \big( F_t(s) \cap V_{ij}^s \big) \biggr).$$

If we set $\tilde{W}_t[l]:= \pi^{-1}(W_t[l])$, which is diffeomorphic to $W_t[l]$, we have the decomposition
$$\tilde{W}_t[l] = \biggl( \bigcup_i \big( \tilde{W}_t[l] \cap T_i \big) \biggr) \cup \biggl( \bigcup_{i,j} \big( \tilde{W}_t[l] \cap T_{ij} \big) \biggr),$$
and clearly each $\tilde{W}_t[l] \cap T_i$ (resp. $\tilde{W}_t[l] \cap T_{ij}$) is a fibre bundle over $L(\Sigma_l)$ with fibre 
$F_t(s) \cap V_i^s$ (resp. $F_t(s) \cap V_{ij}^s$). 

In Lemma \ref{lemma}, we showed that each part of the Milnor fibre of type $F_t(s) \cap V_{ij}^s$ is a finite disjoint union of cylinders. Hence $\tilde{W}_t[l] \cap V_{ij}$ is a fibre bundle 
over $L(\Sigma_l)$
with fibre a finite disjoint union of cylinders. The classification of such kind of fibrations yields that $\tilde{W}_t[l] \cap T_{ij}$ is a Seifert manifold.

Now, if $V_i^s \subset X_s$ is such that $a_i=b_i$, by $(iv)$ of Lemma \ref{lemma} we have that $F_t(s) \cap V_i^s$ is a disjoint union of 
cylinders. As before, we conclude that the corresponding piece $\tilde{W}_t[l] \cap T_i$ is a Seifert manifold.

If $V_i^s \subset X_s$ is such that $a_i \neq b_i$, by $(ii)$ of Lemma \ref{lemma} we have that $F_t(s) \cap V_i^s$ is a finite covering over $A_i^s \cap V_i^s$, which is a 
disjoint union of multi-punctured disks. The corresponding piece $\tilde{W}_t[l] \cap T_i$ is a finite covering of $A_i \cap T_i$.

If $A_i^s$ does not represent a set of rupture vertices of the dual graph of $A^s$ then $F_t(s)\cap T_i$ is a finite disjoint union of either discs or cylinders, and therefore 
$\tilde{W}_t[l]\cap T_i$ is a Seifert manifold.

If $A_i^s$ represents a set of rupture vertices of the dual graph of $A^s$
we claim that $A_i \cap T_i$ is a Waldhausen manifold. If the claim is true then 
$\tilde{W}_t[l] \cap T_i$ is a Waldhausen manifold by Corollary~\ref{cor_wald}. Let us prove the claim.

For any point of $s\in W[l]\setminus\{O\}$ we consider the smooth complex surface germ $(\Omega_s,s)$, where $\Omega_s:=X\cap \{z_1=z_1(s)\}$. The family of functions 
$fg|_{\Omega_s}$ is $\mu$-constant, and hence has a simultaneous resolution over the punctured disc $W[l]\setminus\{0\}$. There is a unique irreducible component $\mathcal{A}$ of the 
exceptional divisor of the simultaneous resolution such that its restriction over the circle $\BS^1$ coincides with $A_i$.
Denote by $\sigma:\mathcal{A}\to W[l]\setminus\{O\}$ the corresponding fibration. The fibres are isomorphic to finitely many $\mathbb{P}^1$'s. Replacing the base by a finite cover
we may assume that the fibre is a single $\mathbb{P}^1$. Since $A_i^s$ represents a set of rupture vertices of the dual graph of $A^s$, the set of intersections of $A_i^s$ with 
the other irreducible components of the total transform of the curve $fg=0$ has cardinality at least $3$, and when we vary $s$ it defines a set of possibly multivalued holomorphic sections of the fibration $\sigma$. 
We have to prove that the complement in $A_i$ of a tubular neighbourhood of these multivalued sections is a Waldhausen manifold. 

After a finite base change we may assume the multi-sections to be uni-valued. The total space of the fibration after the base change is a finite cyclic cover of the original total 
space of the fibration. Thus we may see the original total space of the fibration as a finite free cyclic quotient of the total space obtained after the base change. 
Since the finite free finite quotient of a Waldhausen manifold is Waldhausen (just saturate by the action the set of separating tori), in order to prove that the complement 
in $A_i$ of a tubular neighbourhood of the multivalued sections is a Waldhausen manifold,
we may assume that the sections are actually univalued.

Let $\psi_j:W[l]\setminus\{O\}\to\mathcal{A}$, with $j=1,..,k$ be the set of sections ($k\geq 3$). For any $s$ in the punctured disc $W[l]\setminus\{O\}$ the triple 
$(\psi_1(s),\psi_2(s),\psi_3(s))$ defines a projective reference of the fibre $\mathcal{A}_s$. Denote by $\mathcal{D}$ the completion of the punctured disc $W[l]\setminus\{O\}$ to a disc.
There is a unique trivial $\mathbb{P}^1$-bundle $\bar\sigma:\bar{\mathcal{A}}\to\mathcal{D}$ together with three sections $\bar{\psi}_i$ forming a projective coordinate system of each fibre and extending
the previous one. Let $\bar\psi_j$ be the unique extension of the section $\psi_j$ to $\mathcal{D}$. 

Consider the collection of points $\{p_1,...,p_d\}:=\{\bar\psi_j(0):1\leq j\leq k\}$ (several values may coincide). We have a fixed isomorphism 
$\bar{\mathcal{A}} \cong \mathbb{P}^1\times\mathcal{D}$. We fix mutually disjoint disks $K_j$ around $p_j$ in $\mathbb{P}^1$. Shrinking the size of $\mathcal{D}$ we may assume that the 
sections hit the boundary of the polydisks $K_i\times\mathcal{D}$ in the interior of the face $K_i\times\partial\mathcal{D}$. Notice the obvious equality $A_i=\bar\sigma^{-1}(\partial\mathcal{D})$
and consider the splitting
$$A_i=(\cup_{j=1}^dK_i\times\partial\mathcal{D}) \bigcup \big( (\mathbb{P}^1\setminus(\cup_{j=1}^dK_i))\times\partial\mathcal{D} \big).$$
The last piece is Seifert and the complement in each of the pieces $K_i\times\partial\mathcal{D}$ of the tubular neighbourhoods of the sections meeting it is Waldhausen, for being the complement of an algebraic link.

Then we have proved $(ii)$ of Theorem \ref{theo_main}.

\section{The Zeta Function of the Monodromy}
\label{section_5}

We want to give a formula to calculate the zeta function of the monodromy $h$ of $\fgbar: (\bC^2,0) \to (\bC,0)$ with an isolated 
critical point, in terms of the combinatorics of the embedded resolution of $fg$, in the same spirit of the work of A'Campo \cite{AC}, where he calculated the monodromy of holomorphic functions.

If we set $F_\theta:= (\fgbar)^{-1}(\eta e^{i \theta})$, for $\eta$ sufficiently small then for each $q \geq 0$, the monodromy $h: F_\theta \to F_\theta$ defines an isomorphism between vector spaces (the cohomology groups) given by $h^*: H^q(F_\theta; \bC) \to H^q(F_\theta; \bC)$, the so called {\it algebraic monodromy}. The {\it zeta function} of such monodromy is defined by
$$Z(t) = \prod_{q\geq0} \big( det (I^* - t h^*) \big) ^{(-1)^{q+1}},$$
where $I^*$ is the identity $(q \times q)$-matrix.

There is a classical way of calculating the zeta function of $h$ in terms of the Lefschetz numbers of $h$ (\cite{We}, \cite{MO}), as follows:

For each $k \geq 1$, the Lefschetz number of the $k$-iteration of $h$ is defined by
$$\Lambda(h^k) = \sum_{q\geq0} (-1)^q trace [(h^*)^k: H^q(F_\theta,\bC) \to H^q(F_\theta,\bC)].$$ 
If we define the integers $s_1, s_2, \dots$ by the relations
$$\Lambda(h^k) = \sum_{i \mid k} s_i,$$
then the zeta function of $h$ is given by
$$Z(t) = \prod_{i\geq1} (1-t^i)^{-s_i/i}.$$

So all we have to do is to calculate the Lefschetz numbers of $h$. First we recall the following lemma:

\begin{lemma}
Consider the following commutative chain map on an exact sequence:
\begin{equation*}
\begin{CD} 
0 @>>> G_0 @>{\a_0}>> G_1 @>{\a_1}>> \dots @>{\a_{n-1}}>> G_n @>>> 0  \\ 
& & @VV{\varphi_0}V @VV{\varphi_1}V & & @VV{\varphi_n}V \\ 
0 @>>> G_0 @>{\a_0}>> G_1 @>{\a_1}>> \dots @>{\a_{n-1}}>> G_n @>>> 0   
\end{CD}
\end{equation*}
Then
$$\sum_{i=0}^n (-1)^i {\rm trace}[\varphi_i] =0.$$
\end{lemma}

\s
Let $\pi:\tilde{M}\to\bC^2$ be an embedded resolution of the germ of curve $fg$ at the origin. Let $E=\cup_{i=1}^s E_i$ be a decomposition of the exceptional divisor of $\pi$ into irreducible components. Let $a_i$ and $b_i$ denote the multiplicity of $E_i$ in the total transforms of the curves $f$ and $g$, respectively. 

Now fix $\theta$ and apply the previous lemma to the Mayer-Vietoris sequence associated to the decomposition of the Milnor fibre $F_\theta$ of $\fgbar$ induced by the boxes $V_i$ and $V_{ij}$ as in Lemma \ref{lemma} (here, in order to simplify notation, we do not make distinction between $V_i$ and $\tilde{V_p}$ nor between $V_{ij}$ and $\tilde{V_{ip}}$). Then we get:

$$\Lambda(h^k) = \sum_{\substack{i=1 \\ a_i \neq b_i}}^s \Lambda(h_{V_i\cap F_\theta}^k) + 
\sum_{\substack{i=1 \\ a_i = b_i}}^s \Lambda(h_{V_i \cap F_\theta}^k) + 
\sum_{\substack{i,j=1 \\ i \neq j}}^s \Lambda(h_{V_{ij}\cap F_\theta}^k) - 
\sum_{i=1}^s \Lambda(h_{\partial V_i\cap F_\theta}^k),$$
where $h_{V_i \cap F_\theta}^k$ (resp. $h_{V_{ij} \cap F_\theta}^k$ and $h_{\partial V_i \cap F_\theta}^k$) denotes the restriction of $h^k$ to the intersection $V_i \cap F_\theta$ (resp. $V_{ij} \cap F_\theta$ and $\partial V_i \cap F_\theta$).

\begin{lemma} 
\label{lema2}
For each $k \geq 1$ we have:
\begin{itemize}
\item[$(a)$] $\Lambda(h_{V_{ij}\cap F_\theta}^k) =0$, for any $i \neq j$; 
\item[$(b)$] $\Lambda(h_{V_i \cap F_\theta}^k) =0$, if $a_i=b_i$; 
\item[$(c)$] $\Lambda(h_{\partial V_i\cap F_\theta}^k) =0$, for any $i$.
\end{itemize}
\end{lemma}

\begin{proof}
We have proved in Lemma \ref{lemma} that each piece of the Milnor fibre considered in this lemma is a disjoint finite union of either cylinders (cases $(a)$ and $(b)$) or circles (case $(c)$). The cylinders of cases $(a)$ and $(b)$ have always a circle of case $(c)$ as a boundary component. Since a cylinder can be retracted to its boundary, it is enough to prove the result for case $(c)$.
The finite union of circles $\partial V_i \cap F_\theta$ is a finite covering over a circle in $E_i$ and the monodromy is compatible with the covering projection. Since the Euler characteristic of the circle vanishes, the result follows from the lemma bellow.
\end{proof}

\begin{lemma}
\label{lemma3}
Let $\pi: X \to B$ be a $m$-covering of a compact manifold with boundary and let $h$ be an automorphism of $X$ such that the diagram
$$
\xymatrix { 
X \ar[r]^h \ar[d]_{\pi} & X \ar[dl]^{\pi} \\ 
B &  \\
}
$$
commutes. For $b \in B$, denote by $h_b: \pi^{-1}(b) \to \pi^{-1}(b)$ the permutation induced by $h$. Then $\Lambda(h^k) = \chi(B) . \Lambda(h_b^k)$. Moreover, if the permutation $h_b$ is cyclic and transitive then 
$$\Lambda(h^k) = 
\begin{cases}
\chi(B) . m \ , \ {\text if} \ m \mid k; \\
\ \ \ \ 0 \ \ \ \ \ , \ {\text if} \ m \nmid k. \\
\end{cases} 
$$
\end{lemma}

\begin{proof} 
Suppose that $B$ is contractible. Then $X \stackrel {\rm homeo}{\simeq} B \times \pi^{-1}(b)$ and hence $H^q(X) \stackrel {\rm isom}{\simeq} H^q(\pi^{-1}(b))$, and therefore $\Lambda(h^k) = \Lambda(h_b^k)$.
If $B$ is not contractible, we can write it as a finite union of contractible sets $B_i$, with $i \in \{1, \dots, c\}$, such that $B_i \cap B_j$ is contractible, for any $i,j \in \{1, \dots, c\}$. Then we proceed by induction on $c$: Supposing that the result is true for $c-1$, we define $\hat{B}= \cup_{i=1}^{c-1}B_i$, $X_{cup}= \pi^{-1}(\hat{B})$ and $X_c = \pi^{-1}(B_c)$. Then $B= \hat{B} \cup B_c$. Applying the Mayer-Vietoris sequence associated to this decomposition one gets 
$$\Lambda(h) = \Lambda(h_{| X_{cup}}) + \Lambda(h_{| X_c}) - \Lambda(h_{| X_{cup} \cap X_c}) = $$
$$\chi(\hat{B}) \Lambda(h_b^k) + \chi(B_c) \Lambda(h_b^k) - \chi(\hat{B} \cap B_c) \Lambda(h_b^k) = $$
$$\chi(B) \Lambda(h_b^k).$$ 

Now observe that  $\Lambda(h_b^k) = \sum_{q\geq0} (-1)^q trace [(h_b^k)^*: H^q(\pi^{-1}(b)) \to H^q(\pi^{-1}(b))]$, which is the trace of the induced isomorphism $(h_b^k)^*: \underbrace{\bZ \times \dots \times \bZ}_{m - times} \to \underbrace{\bZ \times \dots \times \bZ}_{m - times}$, which is equal to $m$ if $m \mid k$, or zero otherwise.

\end{proof}

Now, if $a_i \neq b_i$, then $V_i\cap F_\theta$ is a covering of degree
$d_i := |a_i-b_i|$ over $E_i$ minus $r_i$-disks, where $r_i$ is the number of double points of the total transform of $fg$ on $E_i$. Moreover, the monodromy $h$ is compatible with the covering projection. Hence using the two previous lemmas we get the following formula:
$$\Lambda(h^k) = \sum_{\substack{i=1 \\ a_i \neq b_i}}^s \Lambda(h_{X_i}^k) =  \sum_{\substack{i=1 \\ a_i \neq b_i \\ d_i \mid k}}^s d_i (2-r_i).$$
Since for each $k \geq 1$ we have
$$\sum_{d_i \mid k} d_i (2-r_i) = \ \Lambda(h^k) \ = \sum_{d_i \mid k} s_{d_i},$$
it follows that
$$s_{d_i} = d_i (2-r_i)$$
and then we have the following theorem:

\begin{theo} \label{funczeta}
Let $f,g: (\bC^2,0) \to (\bC,0)$ be two holomorphic function-germs such that the real analytic map-germ $\fgbar: (\bC^2,0) \to (\bC,0)$ has an isolated critical point. Let $\pi:\tilde{M}\to\bC^2$ be an embedded resolution of the germ of curve $fg$ at the origin. Let $E=\cup_{i=1}^s E_i$ be a decomposition of the exceptional divisor of $\pi$ into irreducible components. Let $a_i$ and $b_i$ denote the multiplicity of $E_i$ in the total transforms of $f$ and $g$, respectively. Set $d_i:=|a_i-b_i|$ and let $r_i$ be the number of double points of the total transform of $fg$ in $E_i$. The the zeta function of the monodromy of the Milnor fibration of $\fgbar$ is given by
$$Z(t) = \prod_{i=1}^s (1-t^{d_i})^{r_i-2}.$$
\end{theo}

\begin{example} {\rm
Consider the holomorphic functions $f(x,y) = x^2+y^3$ and $g(x,y) = x^3+y^2$. Then the real analytic map-germ $\fgbar: (\bR^4,0) \to (\bR^2,0)$ has an isolated singularity at $0 \in \bC^2$. The graph of the resolution of the complex curve $fg$ with the corresponding multiplicities is:

$$
\xymatrix{ 
	     &	*{} \ar@{-}[ddd] &	& & & *{} \ar@{-}[ddd] &      \\	
*{}\ar@{-}[r]^{E_2 \ (3,2)} & *{} \ar@{-}[r] & *{} & & *{}\ar@{-}[r]^{E_4 \ (2,3)} & *{} \ar@{-}[r] & *{}    \\
*{}\ar@{-}[r] & *{} \ar[r]^f & *{} & & *{}\ar@{-}[r] & *{} \ar[r]^g & *{}    \\
*{}\ar@{-}[rrrrrr]^{E_1 \ (2,2)} & *{} \ar@{-}[d]_{E_3 \ (6,4)} & & & & *{} \ar@{-}[d]^{E_5 \ (4,6)} & *{}  \\
 & *{} & & & & *{} & \\
}
$$

\s
In the holomorphic case $fg$, the part of the Milnor fibre $F_\eta$ inside each box of type $V_{ij}$ is a disjoint union of $gcd(a_i+b_i, a_j+b_j)$ cylinders, and the part of the Milnor fibre inside each box of type $V_i$ is an $(a_i+b_i)$-covering of a sphere minus $r_i$ disks, with Euler characteristic $(a_i+b_i)(2-r_i)$. Then: $F_\eta \cap V_{13}$ (and $F_\eta \cap V_{15}$) is two cylinders; $F_\eta \cap V_{23}$ (and $F_\eta \cap V_{45}$) is five cylinders; $F_\eta \cap V_1$ is two cylinders; $F_\eta \cap V_2$ (and $F_\eta \cap V_4$) is five disks; $F_\eta \cap V_3$ (and $F_\eta \cap V_5$) is a compact surface of genus $2$ with boundary eight circles. Hence the Milnor fibre of $fg$ is a twice-punctured surface of genus $5$.

In the real analytic case $\fgbar$, according to Lemma \ref{lemma}, we have that: $F_\eta \cap V_{13}$ (and $F_\eta \cap V_{15}$) is two cylinders; $F_\eta \cap V_{23}$ (and $F_\eta \cap V_{45}$) is a cylinder; $F_\eta \cap V_1$ is two cylinders; $F_\eta \cap V_2$ (and $F_\eta \cap V_4$) is a disk; $F_\eta \cap V_3$ (and $F_\eta \cap V_5$) is a compact surface of genus $0$ with boundary four circles. Hence the Milnor fibre of $\fgbar$ is a twice-punctured surface of genus $1$, that is, a torus with boundary two disjoint circles.

Moreover, by Theorem \ref{funczeta} the zeta function of $h^{fg}$ is given by
$$Z(t) = (1-t^5)^{-2} (1-t^{10})^2$$
and the zeta function of $h^{\fgbar}$ is given by
$$Z(t) = (1-t)^{-2}(1-t^2)^2.$$
}
\end{example}


\end{document}